\documentclass{article}
\usepackage{aarons_packages}
\title{Asymptotic analysis of harmonic functions in singular domains with inhomogenous Robin boundary conditions.}
\author{Aaron Pim, Kirill Cherednichenko and Jey Sivaloganathan}
\date{September 2020}
\begin{document}

\maketitle
\begin{abstract}
In 1991, Vladimir Maz'ya, Serguei Nazarov and Boris Plamenevskij developed the theory of compound asymptotics for elliptic boundary value problems in singularly perturbed domains. They considered a harmonic function whose domain contains a small inclusion. We applied this technique in the analysis of a Nematic liquid crystal with a small colloidal inclusion. However, we realised that the Maz'ya, Nazarov and Plamenevskij did not consider the asymptotic analysis of Robin boundary conditions, which corresponded to weak anchoring in the context of liquid crystals. 

In this piece we shall derive an asymptotic approximation to a harmonic function, in a domain with a small circular inclusion of radius $\epsilon>0$, with inhomogenous Robin boundary conditions and a corresponding parameter $\kappa>0$. We shall then prove that the difference between the exact solution and the approximation is uniformly bounded and derive the rate as a function of $\epsilon$ and $\kappa$.
\end{abstract}
\section{Introduction}
The liquid crystalline phase is an intermediate state of matter between isotropic liquid and a solid crystal. A material is said to be in the liquid crystalline phase if it displays certain properties typically associated with both the isotropic liquid phase and the solid crystalline phase \cite{Virga}. We shall be considering the Nematic liquid crystal phase, where the average orientation of the molecules can be modelled as a unit vector field \cite{IntroOptics, DeGennes}.

A colloidal suspension is a type of soft matter which tiny particulates, the colloidal inclusions, are suspended in a distinct host media; some every day examples of colloidal suspensions are smoke, fog, paint and milk \cite{Topological_colloids}. These colloids are typically much larger, of order of magnitude $\mu m$ \cite{Colloidal_Inculsions}, than the surrounding liquid crystal but much smaller than the domain. 

In the field of optics, liquid crystal displays with solid colloidal inclusions (also known as composite liquid crystal displays) have been of particular interest in recent times as they can change their scattering properties via external fields \cite{Homogenisation_GL_w_Colloids}. A good example would be an assemblage of silica spheres whose configurations are easily erased, an important property in liquid crystal displays \cite{Erasable_Optical_Storage_Bistable_LC_Cells}.

In this piece we wish to understand how the presence of a single colloidal inclusion $D_{\epsilon}\subset \mathbb{R}^2$ affects the configuration of the nematic liquid crystal in a shallow well $\Omega\subset \mathbb{R}^2$. We shall use asymptotic analysis to construct an approximation to the solution, with the small asymptotic parameter being the radius of our colloidal inclusion, which shall be denoted $\epsilon \ll 1$.

In broad terms, asymptotic analysis is the construction of a series approximate solutions to problem with a parameter which is either very large or very small \cite{Asymptotic_Analysis}. If the boundaries of the domain are dependent on this asymptotic parameter and the limit boundary is not smooth; this is known as a singular perturbed boundary \cite{Nazarov} and requires careful analysis.

The specific method we shall derive is known as \textit{compound asymptotics}, where we approximate the solution $u^\epsilon:\Omega \setminus D_\epsilon \rightarrow \mathbb{R}$ by an \textit{interior} function $v_0: \Omega \rightarrow \mathbb{R}$, which satisfies the same underlying equation and the boundary conditions on $\partial \Omega$, similar analyses sometimes refer to this function as the \textit{outer} function. As $v_0$ does not satisfy the boundary conditions on $\partial D_\epsilon$ we consider an \textit{exterior} function $w_0:\mathbb{R}^2 \setminus D_\epsilon$, which satisfies the same boundary conditions on $\partial D_\epsilon$ as the function $u^\epsilon - v_0$, and is such that $w_0|_{\partial \Omega}\approx 0$, thus we approximate the function $u^\epsilon$ by the sum $v_0+w_0$. This method is referred to as compound asymptotic because it can be applied algorithmically and recursively, as the same analysis may be applied to the function $u^\epsilon - v_0 - w_0$ to obtain functions $v_1$ and $w_1$.

The underlying equations that we shall eventually be considering is the scalar Laplace equation, thus we ask ourselves "what boundary conditions does our function satisfy?". The asymptotic analysis of a harmonic function with Dirichlet boundary conditions on a singularly perturbed boundary, was analysed by Maz'ya, Nazarov and Plamenevskij  \cite{Nazarov}, in the context of liquid crystals this corresponds to strong anchoring \cite{Apala_bistable}. However, compound asymptotic analysis of Robin boundary conditions has not been conducted, such boundary conditions correspond to weak anchoring in the context of liquid crystals. We shall derive our problem from the one constant approximation to the Oseen-Frank energy density \cite{OseenFrank_limit_and_beyond,Frank} with the Rapini-Papoular anchoring energy \cite{Rapini-Papoular}.
\section{Derivation of the problem}
\subsection{Domain definition}
We define the well in which the liquid crystal and the colloidal inclusions occupy to be $\Omega \subset \mathbb{R}^2$ open and bounded with Lipschitz continuous boundary $\partial \Omega$. Let the centre of our colloidal inclusion be given by $\mathbf{c}\in \Omega$, and we define the constants $0< R_{min}< R_{max}$ by 
\begin{equation}\label{Rmin and Rmax definition}
    \min\limits_{\mathbf{x}\in \partial \Omega}|\mathbf{x}-\mathbf{c}| =: R_{min}\leqslant |\mathbf{x}-\mathbf{c}| \leqslant R_{max}:= \max\limits_{\mathbf{x}\in \partial \Omega}|\mathbf{x}-\mathbf{c}|, \quad \forall \mathbf{x}\in\partial \Omega.
\end{equation}
Finally we assume that the radius of our inclusion, the asymptotic parameter, satisfies
\begin{equation}\label{epsilon bound}
        \epsilon \leqslant \frac{R_{min}}{2}, \quad \Omega_\epsilon := \Omega \setminus D_\epsilon, \quad D_\epsilon := \overline{B}_{\epsilon}(\mathbf{c}).
\end{equation}
We denote the outward pointing normal by $\mathbf{n}:\partial \Omega_\epsilon\rightarrow \mathbb{S}^1$.
\subsection{Energy functional and first variation}
Consider the one-constant approximation Oseen-Frank energy density with Rapini-Papoular surface anchoring energy, which is denoted $E_{\kappa}:H^1(\Omega, \mathbb{S}^1)\rightarrow [0,\infty)$ and given by
\begin{equation}\nonumber
	E_{\kappa}[\mathbf{u}] = \int\limits_{\Omega_{\epsilon}}\left(|\nabla u_1|^2 +|\nabla u_2|^2\right) \ d\mathbf{x} + \frac{1}{2\kappa}\int\limits_{\partial \Omega_\epsilon} (1-(\mathbf{u}\cdot \mathbf{p}^{\epsilon})^2) \ dS_{\mathbf{x}},
\end{equation}
where the unit vector $\mathbf{u}=(u_1,u_2)$ represents the average direction of molecular orientation, the scalar $\kappa >0$ represents the extrapolation length \cite{Extrapolation_length} and the unit boundary vector field $\mathbf{p}^{\epsilon}\in L^\infty(\partial \Omega_{\epsilon}, \mathbb{S}^1)$ represents the preferred direction of orientation along the boundary. The extrapolation length $\kappa$ shall become one of the main variables in the analysis as we shall investigate how well the asymptotic approximation matches the original solution as $\epsilon$ and $\kappa$ vary. We have by the Babuška–Lax–Milgram theorem \cite{Evans} that the functional $E_{\kappa}$ has a unique minimiser, thus we shall consider the first variation
\begin{equation}\nonumber
        	\begin{split}
        		0 =\delta E_{\kappa} = \int\limits_{\Omega_{\epsilon}}\left (\nabla \underline{\phi} \cdot \nabla \mathbf{u}- \underline{\phi}\cdot \mathbf{u}|\nabla\mathbf{u}|^2\right)\ d\mathbf{x} - \frac{1}{\kappa}\int\limits_{\partial \Omega_{\epsilon}}(\mathbf{u}\cdot \mathbf{p}^\epsilon)((T\mathbf{u})\cdot \mathbf{p}^\epsilon)((T\mathbf{u})\cdot \underline{\phi})\ dS_{\mathbf{x}}.
        	\end{split}
\end{equation}
for all $\underline{\phi} \in C^{\infty}(\bar{\Omega},\mathbb{R}^2) $ where $T = \left(\begin{array}{cc}0 & -1 \\ 1 & 0 \end{array} \right)$ is the $\frac{\pi}{2}$ rotation matrix. 
\subsection{Angular formulation and assumptions}
It is at this point we notice an issue, the analysis we wish to conduct is for a scalar function $u^\epsilon$, whereas the above weak partial differential equation is in terms of a unit vector field $\mathbf{u}$, thus we wish to pass to the angular formulation of the problem. To ensure that the angular formuation $u^\epsilon$ is of sufficient regularity we must assume that $\int\limits_{C} u_1 d u_2 - u_2 d u_1 = 0$ for all closed continuous curves $C \subset \Omega_{\epsilon}$. Thus we may define the gradient of $u^\epsilon$ in terms of $\mathbf{u}$,
$$\nabla u^\epsilon := u_1\nabla u_2 - u_2 \nabla u_1, \text{ on } \overline{\Omega}_\epsilon. $$
As we have passed to the angular formulation of the problem, we must do the same for the boundary data. Let the boundary vector $\mathbf{p}^{\epsilon}$  be such that there exist scalar functions $f_{\Omega}, g_{\Omega} \in L^\infty(\partial \Omega)$, $f_{D} \in L^\infty([-\pi,\pi))$ and $g_D\in L^\infty(\partial D_{\epsilon})$ such that
\begin{equation}\nonumber
	\mathbf{p}^\epsilon:=\left(\begin{array}{c}
	\cos (f+\epsilon g) \\ \sin (f+\epsilon g)
	\end{array} \right) \quad f(\mathbf{x}):=\left\{\begin{array}{ll}
	f_{\Omega}(\mathbf{x}), & \mathbf{x}\in \partial \Omega \\
    f_{D}(\vartheta(\mathbf{x}-\mathbf{c})), & \mathbf{x}\in \partial D_\epsilon
	\end{array} \right. \quad g(\mathbf{x}):=\left\{\begin{array}{ll}
	g_{\Omega}, & \text{ on }\partial \Omega \\
    g_{D}, & \text{ on } \partial D_\epsilon
	\end{array} \right.
\end{equation}
where $\vartheta: \mathbb{R}^2\setminus \{\underline{0}\}\rightarrow [-\pi,\pi)$ is given by $\vartheta(\mathbf{x}):=\text{atan2}(\omega_2,\omega_1)$, we shall use $\vartheta$ to refer to the angular co-ordinate of the vector $\mathbf{x}-\mathbf{c}$. As we are considering harmonic functions and with a circular boundary, this motivates us to consider a Fourier expansion of $f_D$, so that we may explicitly express the analytical solution to our exterior problem. We require that there exists a uniformly convergent Fourier series
\begin{equation}\label{Fourier series boundary data}
	f_D(\vartheta) := \frac{1}{2\pi}\int\limits_{-\pi}^{\pi}f_D(\vartheta) d\vartheta + \sum\limits_{n=1}^{\infty} \mathbf{f}^n\cdot  \left(\begin{array}{c} \cos n\vartheta \\ \sin n\vartheta\end{array} \right),
\end{equation}
for $\mathbf{f}^n \in \mathbb{R}^2$ for $n \in \mathbb{N}$.
\subsection{Strong problem and regularity}
When we pass to the angular formulation in our previous problem we deduce that $u^\epsilon \in H^1(\Omega_\epsilon)$ satisfies Laplace's equation, in the weak sense. However, we may obtain a higher degree of regularity by applying Weyl's lemma\cite{Weyls_Lemma} to deduce that it is in the class$ (H^1\cap C^{\infty})(\Omega_{\epsilon})$. We now turn our attention to the boundary condition. as $u^\epsilon$ satisfies $$ \frac{\partial u^\epsilon}{\partial \mathbf{n}} =  \frac{1}{2\kappa}\sin(2(f+\epsilon g - u^{\epsilon})),\text{ on }\partial \Omega_\epsilon.$$
This boundary condition, in general, will not yield an explicit solution thus we wish to approximate this boundary condition by linearising the sine term. Thus we have that $u^\epsilon \in (H^1\cap C^{\infty})(\Omega_{\epsilon})$ satisfies
$$ -\Delta u^{\epsilon} = 0,  \text{ on } \Omega_\epsilon, \quad 
    u^{\epsilon}+ \kappa \frac{\partial u^{\epsilon}}{\partial \mathbf{n}} = f+\epsilon g, \text{ on }\partial \Omega_\epsilon, $$
 we seek to asymptotically approximate this function. To prove the asymptotic bounds for our approximation, we shall apply the maximum principle for harmonic functions. Thus we must prove that the trace of $u^\epsilon$ is bounded, so that we may apply the maximum principle.
\begin{lemma}
	The trace of $u^\epsilon|_{\partial \Omega_\epsilon}$ and the normal derivative $\frac{\partial u^\epsilon}{\partial \mathbf{n}}$ are both bounded.
\end{lemma}
\begin{proof}
	 Using Sobolev trace  theorems \cite{Orane} we know that $u^\epsilon|_{\partial \Omega_\epsilon} \in H^{\frac{1}{2}}(\partial \Omega_\epsilon) $ \cite{Hitchhikers_Guide_to_Sobolev}, which implies that $u^\epsilon|_{\partial \Omega_\epsilon} \in L^{2}(\partial \Omega_\epsilon)$. As the functions $f,g$ were assumed to be of class $L^\infty$ this implies that $f, g \in L^2(\partial \Omega_\epsilon)$, consequently the boundary conditions of $u^\epsilon$ imply that $\frac{\partial u^\epsilon}{\partial \mathbf{n}}\in L^2(\partial \Omega_\epsilon)$.
     
     We now use the trace theorems to deduce that $u^\epsilon \in H^{\frac{3}{2}}(\Omega_\epsilon)$; reapplying the Dirichlet trace theorem we deduce that $u^\epsilon|_{\partial \Omega_\epsilon} \in H^{1}(\partial \Omega_\epsilon)$. 
     
     Finally, we apply Morrey's inequality \cite{Jeys_Sobolov_Spaces_Course} for a 1 dimensional space to deduce that $u^\epsilon|_{\partial \Omega_\epsilon} \in C(\partial \Omega_\epsilon)$. As $\partial \Omega_{\epsilon}$ is closed and bounded, this implies boundedness, $u^\epsilon|_{\partial \Omega_\epsilon} \in L^\infty(\partial \Omega_\epsilon)$; applying the boundary condition we deduce that $\frac{\partial u^\epsilon}{\partial \mathbf{n}} \in L^\infty(\partial \Omega_\epsilon)$.
\end{proof}
As previously mentioned, we wish to apply the maximum principle to prove the bound of our problem; however, the maximum principle \cite{Evans} is in terms of the Dirichlet trace, whereas our boundary conditions are Robin. Thus we require a maximum principle in terms of the Robin trace.
\subsection{Maximum principle for Robin boundary conditions\label{Bounded Robin Maximum Principle}}
                Whilst we shall be applying this lemma to the function $u^\epsilon$, the maximum principle we are about to derive does not rely on the specific geometry $\Omega_\epsilon$. Thus we consider a more general domain $A\subset \mathbb{R}^2$  which is bounded with sufficiently smooth boundary $\partial A$ and outward pointing normal $\mathbf{n}:\partial A\rightarrow \mathbb{S}^1$. Let $a \in C^\infty(A)\cap C(\overline{A})$ be harmonic and satisfy 
                $$ \kappa \frac{\partial a}{\partial \mathbf{n}} + a = b, \text{ on } \partial A,$$
                for a given $b\in L^\infty(\partial A)$. 
                \begin{lemma}
                	The function $a$ satisfies 
                    \begin{equation}\label{Dirichlet maximum principle}
                    \inf_{\mathbf{y}\in \partial A} b(\mathbf{y}) \leqslant a(\mathbf{x}) \leqslant \sup_{\mathbf{y}\in \partial A} b(\mathbf{y}), \quad \forall \mathbf{x}\in \overline{A}.
                \end{equation}
                \end{lemma}               
            \begin{proof}
                Applying the maximum principle for harmonic functions \cite{Evans} and the continuity of $a$ on the boundary we deduce that
                \begin{equation}\nonumber
                    \inf\limits_{\mathbf{y}\in \partial A} a(\mathbf{y}) \leqslant a(\mathbf{x}) \leqslant \sup\limits_{\mathbf{y}\in \partial A} a(\mathbf{y}), \quad \forall \mathbf{x}\in \overline{A}.
                \end{equation}
                We shall now consider the bounds on the normal derivatives. As the domain $A$ is bounded, we may apply the divergence theorem
                \begin{equation}\nonumber
                    0 = \int\limits_{A}\Delta a\ d\mathbf{x}= \int\limits_{\partial A} \frac{\partial a}{\partial \mathbf{n}}\ dS_{\mathbf{x}}.
                \end{equation}
               Additionally, we have that following inequalities,
                \begin{equation}\nonumber
                    |\partial A| \inf\limits_{\mathbf{y}\in \partial A}\left(\frac{\partial a}{\partial \mathbf{n}}\right) \leqslant \int_{\partial A} \frac{\partial a}{\partial \mathbf{n}}\ dS_{\mathbf{x}} \leqslant |\partial A| \sup\limits_{\mathbf{y}\in \partial A}\left(\frac{\partial a}{\partial \mathbf{n}}\right).
                \end{equation}
               From our earlier analysis we know that $\frac{\partial u}{\partial \mathbf{n}} \in L^\infty(\partial A)$, thus the supremum and infimum both exist. Applying the above result and the assumption that $\kappa >0$ we have that
                \begin{equation}\nonumber
                    \inf\limits_{\mathbf{y}\in \partial A}\left(\kappa\frac{\partial u}{\partial \mathbf{n}}\right) \leqslant 0 \leqslant \sup\limits_{\mathbf{y}\in \partial A}\left(\kappa\frac{\partial u}{\partial \mathbf{n}}\right).
                \end{equation}
                Summing the above inequalities with the ones from equation (\ref{Dirichlet maximum principle}) we obtain our result.
            \end{proof}
            We shall use this result later to prove that the difference between our exact solution $u^\epsilon$ and the approximation $v_0+w_0$ is uniformly bounded by a constant depending on $\epsilon$ and $\kappa$.
\section{Derivation of main result}
In this section, demonstrate the construction of our asymptotic method, in a similar manner to Maz'ya et al. \cite{Nazarov}. We shall initially begin with a demonstrating an incorrect method, choosing what many would assume to be the natural choice for the interior solution. However, such a choice will have consequences when we consider the corresponding exterior problem, as we shall discover that our problem is over-determined and we require an additional degree of freedom. This shall be a consequence of the fact that we are in two dimensions.

The incorrect derivation is important because it motivates key aspects of the asymptotic method: the correct problem for the interior solution to solve and a theorem regarding the decay of exterior harmonic functions with Robin boundary conditions, in two dimensions. We shall denote the incorrect solution by $V_0$ as to differ it from the correct solution $v_0$, although we shall utilise $V_0$ in the definition of $v_0$.
\subsection{An intentionally incorrect derivation}
Recall that the method of compound asymptotics requires us to consider an interior function $V_0:\Omega\rightarrow \mathbb{R}$ which satisfies the leading order term of the boundary conditions on $\partial \Omega$ and the same underlying equation in $\Omega_\epsilon$. A natural choice would be to consider a $V_0 \in (H^1\cap C^{\infty})(\Omega)$ which satisfies the following boundary value problem
\begin{equation}\label{Interior boundary condition}
	-\Delta V_0 = 0, \text{ on }\Omega, \quad V_0 +\kappa \frac{\partial V_0}{\partial \mathbf{n}} = f_{\Omega}, \text{ on }\partial\Omega.
\end{equation}
To motivate the boundary conditions for the exterior solution we must consider the discrepancy between the exact solution $u^\epsilon$ and the interior solution $V_0$. The function $u^\epsilon-V_0$ is harmonic in $\Omega_\epsilon$ and satisfies the following boundary conditions
\begin{equation}\nonumber
	\left(I+ \kappa \frac{\partial }{\partial \mathbf{n}}\right)(u^\epsilon-V_0) = \left\{\begin{array}{ll}
	\epsilon g_\Omega, & \text{ on }\partial \Omega \\
    f_D(\vartheta(\cdot - \mathbf{c})) - V_0 + \kappa \frac{\partial V_0}{\partial r} + \epsilon g_D, & \text{ on }\partial  D_\epsilon\\
	\end{array} \right.
\end{equation}
where $r:=|\mathbf{x}-\mathbf{c}|$. The exterior solution $W_0$ must be harmonic, satisfy the same Robin boundary conditions on $\partial D_\epsilon$ (up to the leading order) and be such that the Robin trace on $\partial \Omega$ is of order $\epsilon$. This final condition is equivalent to the following decay condition
\begin{equation}\nonumber
	W_0 + \kappa \frac{\partial W_0}{\partial r}\rightarrow 0 \quad \text{ as }r\rightarrow \infty,
\end{equation}
however, as we have assumed that $w_0$ is bounded this implies that $\frac{\partial W_0}{\partial r}\rightarrow 0$ as $r\rightarrow \infty$, hence the above condition is reduced to 
\begin{equation}\label{exterior decay condition}
	W_0 \rightarrow 0 \quad \text{ as }r\rightarrow \infty.
\end{equation}
As harmonic functions are real analytic, this implies that we may consider the following Taylor expansion for the function $V_0$
$$ V_0(\mathbf{x}) = V_0(\mathbf{c}) + \nabla V_0(\mathbf{c})\cdot (\mathbf{x}-\mathbf{c}) + \frac{1}{2}\frac{\partial \nabla V_0}{\partial \omega_1}(\mathbf{c})\cdot \left(\begin{array}{c}
(x_1-c_1)^2 - (x_2-c_2)^2 \\ 2(x_1-c_1)(x_2-c_2)
\end{array} \right) + \mathcal{O}(|\mathbf{x}-\mathbf{c}|^3).$$
This implies that when we restrict the values of $x$ to the boundary $\partial D_\epsilon$, which corresponds to $|\mathbf{x}-\mathbf{c}|=\epsilon$, we obtain
$$\left.V_0 - \kappa \frac{\partial V_0}{\partial r}\right|_{\partial D_\epsilon} = V_0(\mathbf{c}) - \kappa \nabla V_0(\mathbf{c})\cdot \left(\begin{array}{c}
\cos \vartheta \\ \sin \vartheta
\end{array} \right) +\mathcal{O}(\epsilon(1+\kappa)), \quad \vartheta \in [-\pi,\pi). $$
Truncating this series at the leading order yields the boundary condition of the exterior problem,
\begin{equation}\label{inccorect exterior boundary condition}
	\left(W_0 - \kappa \frac{\partial W_0}{\partial r}\right)(1,\vartheta) = f_D(\vartheta) - V_0(\mathbf{c}) + \kappa \nabla V_0(\mathbf{c})\cdot \left(\begin{array}{c}
\cos \vartheta \\ \sin \vartheta
\end{array} \right)
\end{equation}
However, we immediately encounter an issue, in that a bounded harmonic function which satisfies equation (\ref{inccorect exterior boundary condition}) only, is already uniquely determined. This is a major issue because such a function might not satisfy the decay condition from equation (\ref{exterior decay condition}). Which motivates the question "What is the limiting value of a bounded, exterior, harmonic function in terms of its Robin boundary conditions?"
\subsection{The limit of a bounded, exterior harmonic function\label{limit theorem}}
            Similar to the maximum principle, the result that we shall derive can be applied to a non-circular inclusion, thus we shall denote this inclusion by $\mathcal{D}$ to differentiate it from $D_\epsilon$. 
            \begin{lemma}
                        Let $\mathcal{D}\subset \mathbb{R}^2$ be a bounded, star-shaped, closed domain with sufficiently smooth boundary and inward pointing normal $\mathbf{n}:\partial \mathcal{D}\rightarrow \mathbb{S}^1$, additionally we assume that $\underline{0}\in \mathcal{D}\setminus \partial \mathcal{D}$. Let $a \in C^\infty(\mathbb{R}^2 \setminus \mathcal{D})$ be a bounded harmonic function satisfying $$a + \kappa \frac{\partial a}{\partial \mathbf{n}} = b, \text{ on } \partial \mathcal{D},$$
            for a given $b\in L^\infty(\partial \mathcal{D})$. The function $a$ satisfies
            \begin{equation}\nonumber
                \lim_{|\mathbf{x}|\rightarrow \infty}a(\mathbf{x})\rightarrow  \int_{\partial \mathcal{D}} b H^\kappa \ dS_{\mathbf{x}},
            \end{equation}
            where the bounded harmonic function $H^\kappa \in C^2(\mathbb{R}^2\setminus \mathcal{D})$ satisfies 
            \begin{equation}\nonumber
                 H^\kappa+\kappa\frac{\partial H^\kappa}{\partial \mathbf{n} } = \frac{\partial H}{\partial \mathbf{n} } \text{ on }\partial \mathcal{D},
            \end{equation}
            where the function $H$ is defined by $$H(\mathbf{x}) :=-\frac{1}{2\pi}\log |\mathbf{x}|+ H_0(\mathbf{x}) \text{ on }\mathbb{R}^2\setminus \mathcal{D}\cup \partial \mathcal{D},$$
            where the function $H_0 \in C^2(\Omega \setminus \mathcal{D})$ is bounded, harmonic and satisfies the Dirichlet boundary condition  $H_0(\mathbf{x})=\frac{1}{2\pi}\log |\mathbf{x}| \text{ on }\partial \mathcal{D}.$
        \end{lemma}
        \begin{remark}
        	We can view that function $H$ as the analogue to the Green's function in an exterior domain.
        \end{remark}
        \begin{proof}
            Let $\left<b,H^\kappa \right>_{\partial \mathcal{D}}:= \int_{\partial \mathcal{D}}b H^\kappa \ dS_{\mathbf{x}}$, by substitution we have that
            \begin{equation}\nonumber
                \left<b,H^\kappa \right>_{\partial \mathcal{D}} = \kappa\left<\frac{\partial a}{\partial \mathbf{n}},H^\kappa \right>_{\partial \mathcal{D}}+\left<a,H^\kappa \right>_{\partial \mathcal{D}}.
            \end{equation}
            As we cannot apply the divergence theorem for unbounded domains, we must consider the Divergence theorem in the domain $\lim\limits_{R\rightarrow \infty}B_R(\mathbf{0})\setminus \overline{\mathcal{D}}$, and then apply the fact that $a$ and $H^\kappa$ are both harmonic we have that
            \begin{equation}\nonumber
                \left<\frac{\partial a}{\partial \mathbf{n}},H^\kappa \right>_{\partial \mathcal{D}}=\lim_{R\rightarrow \infty} \int\limits_{\partial B_R} \left(a \frac{\partial H^\kappa}{\partial \mathbf{n}}-\frac{\partial a}{\partial \mathbf{n}}H^\kappa \right)\ dS_{\mathbf{x}} + \left<a,\frac{\partial H^\kappa}{\partial \mathbf{n}}\right>_{\partial \mathcal{D}}
            \end{equation}
            Thus by Green's $2^{\rm nd}$ identity and the harmonic nature of $a$ and $H^\kappa$ we have that
            \begin{equation}\nonumber
               \left<b,H^\kappa \right>_{\partial \mathcal{D}}= \int\limits_{\partial \mathcal{D}} \left( \kappa\frac{\partial H^\kappa}{\partial \mathbf{n}} + H^\kappa\right) a \ dS_{\mathbf{x}} = \int\limits_{\partial \mathcal{D}} \frac{\partial H}{\partial \mathbf{n}} a \ dS_{\mathbf{x}}.
            \end{equation}
            Results from  \cite{Nazarov} state that if $a$ is an exterior bounded harmonic function then it satisfies
            \begin{equation}\nonumber
                \lim\limits_{|\xi|\rightarrow \infty}a(\xi)\rightarrow \int_{\partial \mathcal{D}} \frac{\partial H}{\partial \mathbf{n}} a \ dS_{\mathbf{x}} = \left<b,H^\kappa \right>_{\partial \mathcal{D}}.
            \end{equation}
        \end{proof}
        \subsection{Motivation for the correct interior problem}
        We shall now apply the result we have derived to the circular inclusion. Let $D:= \overline{B}_1(\underline{0})$ denote the circular inclusion in the re-scaled domain given by
        $$\left\{\left.\frac{\mathbf{x}-\mathbf{c}}{\epsilon}\right| \ \forall \ \mathbf{x}\in \Omega_\epsilon \right\} $$
        note that for $\epsilon$ sufficiently small the re-scaled domain can be approximated to be $\mathbb{R}^2\setminus D$. The corresponding $H_0$ function on $\mathbb{R}^2\setminus D$ is identically zero, which implies that $H(\xi) := -\frac{1}{2\pi}\log |\xi|$ for $|\xi| \geqslant 1$; thus the function $H^\kappa$ is bounded, harmonic and satisfies $$H^\kappa - \kappa \frac{\partial H^\kappa}{\partial r}  = -\frac{\partial H}{\partial r} = \frac{1}{2\pi}\text{ for }r = 1.$$ Consequently, $H^\kappa(\xi) = \frac{1}{2\pi}$ for all $|\xi|\geqslant 1$, this is important because it implies that the limit of the bounded harmonic exterior function $a:\mathbb{R}^2\setminus D\rightarrow \mathbb{R}$ is given by
        \begin{equation}\nonumber
        	\lim\limits_{|\xi|\rightarrow \infty}a(\xi) \rightarrow \frac{1}{2\pi}\int\limits_{\partial D}\left(a-\kappa \frac{\partial a}{\partial r} \right)\ dS_{\xi},
        \end{equation}
        which is the average of the boundary data. This implies that the limit of $W_0$, from equation (\ref{inccorect exterior boundary condition}), is given by $$ \lim\limits_{\xi \rightarrow \infty} w_0(\xi)\rightarrow \frac{1}{2\pi}\int\limits_{-\pi}^{\pi} f_D(\vartheta)\ d\vartheta - V_0(\mathbf{c}).$$
        Thus for an arbitrary choice of $f_D$ the decay condition in equation (\ref{exterior decay condition}) is not satisfied. This implies that we require an additional degree of freedom in the exterior problem, we recall that the boundary data for the exterior problem was determined by the interior problem. Thus if we reformulate the interior problem we may resolve this issue.
          \subsection{Solution of the interior problem}
We seek a function $v_0$ whose domain is $\Omega$, which is harmonic in $\Omega_\epsilon$ for all $0 <\epsilon \ll 1$. In the previous interior problem, we assumed that the function $V_0$ was harmonic in the circular inclusion, which resulted in the over-determined exterior problem issue. However, if we allow a singularity at the centre of the inclusion then we introduce an additional degree of freedom to the problem, whilst also satisfying the same underlying equations. We seek a $v_0 \in C^{\infty}(\Omega\setminus \{\mathbf{c}\})$ which satisfies the following boundary value problem
\begin{equation}\nonumber
	-\Delta v_0 = c_0\delta(\mathbf{x}-\mathbf{c}), \text{ on }\Omega, \quad v_0 +\kappa \frac{\partial v_0}{\partial \mathbf{n}} = f_{\Omega}, \text{ on }\partial\Omega,
\end{equation}
where $c_0 \in \mathbb{R}$ is unknown constant to be determined in the exterior problem. As both the boundary conditions and the underlying equation are linear we may express the solution to this problem as a linear combination of a bounded harmonic function, which satisfies the boundary conditions, and a Green's function
\begin{equation}\nonumber
	v_0 := V_0 + c_0 G^\kappa(\cdot, \mathbf{c}), \text{ on }\Omega\setminus \{\mathbf{c}\}
\end{equation}
where $V_0 \in C^\infty(\Omega)$ satisfies
\begin{equation}\nonumber
	-\Delta V_0 = 0 \text{ on }\Omega, \quad V_0 +\kappa \frac{\partial V_0}{\partial \mathbf{n}} = f_{\Omega} \text{ on }\partial\Omega,
\end{equation}
and the function $G^\kappa$ is the Green's function of the Laplace operator with Robin boundary conditions given by
\begin{equation}\label{Greens function definition}
	-\Delta G^\kappa(\mathbf{x},\mathbf{y}) :=\delta(\mathbf{x}-\mathbf{y})\text{ on }\Omega, \quad G^\kappa +\kappa \frac{\partial G^\kappa}{\partial \mathbf{n}} = 0 \text{ on }\partial\Omega.
\end{equation}
Similar to before we may express it as the sum of the Newtonian potential in two-dimensions and a bounded function
$$G^\kappa(\mathbf{x},\mathbf{y}) = -\frac{1}{2\pi}\log |\mathbf{x}-\mathbf{y}| + \mathcal{G}^{\kappa, \mathbf{y}}(\mathbf{x}) $$
where the harmonic function $\mathcal{G}^{\kappa, \mathbf{y}}\in C^\infty(\Omega)$ satisfies
$$ \left(\mathcal{G}^{\kappa, \mathbf{y}} + \kappa \frac{\partial \mathcal{G}^{\kappa, \mathbf{y}}}{\partial \mathbf{n}}\right)(\mathbf{x}) = \frac{1}{2\pi}\log |\mathbf{x}-\mathbf{y}| + \frac{\kappa}{2\pi}\frac{(\mathbf{x}-\mathbf{y})\cdot \mathbf{n}(\mathbf{x})}{|\mathbf{x}-\mathbf{y}|^2}, \quad \mathbf{x}\in \partial \Omega. $$
This decomposition is important because we can utilise the real analytic nature of $\mathcal{G}^{\kappa, \mathbf{y}}$to deduce a leading order approximation to $G^\kappa$ and it's Robin boundary data on $\partial D_\epsilon$
\begin{equation}\nonumber
	G^{\kappa}(\mathbf{x},\mathbf{c}) = -\frac{1}{2\pi}\log |\mathbf{x}-\mathbf{c}| +\mathcal{G}^{\kappa, \mathbf{c}}(\mathbf{c}) + \nabla \mathcal{G}^{\kappa, \mathbf{c}}(\mathbf{c})\cdot (\mathbf{x}-\mathbf{c}) + \mathcal{O}(|\mathbf{x}-\mathbf{c}|^2),
\end{equation}
\begin{equation}\nonumber
	\begin{split}
		\left. G^{\kappa}(\mathbf{x}, \mathbf{c}) - \frac{\partial G^{\kappa}}{\partial r}(\mathbf{x}, \mathbf{c})\right|_{\partial D_\epsilon} = -\frac{1}{2\pi}\log \epsilon +\mathcal{G}^{\kappa, \mathbf{c}}(\mathbf{c})  + \frac{\kappa}{2\pi \epsilon} - \kappa \nabla \mathcal{G}^{\kappa, \mathbf{c}}(\mathbf{c})\cdot \left(\begin{array}{c}
		\cos \vartheta \\ \sin \vartheta
		\end{array} \right) + \mathcal{O}(\epsilon).
	\end{split}
\end{equation}
In the incorrect derivation we considered the discrepancy between the exact solution $u^\epsilon$ and our interior function, to motivate the boundary data of the exterior problem. The same is applied here, on the boundary of $\partial \Omega$ the discrepancy is given by 
$$ \left(I+ \kappa \frac{\partial }{\partial \mathbf{n}}\right)(u^\epsilon-v_0) = \epsilon g_\Omega.$$
Which provides motivation for an identical decay condition, however the discrepancy on the boundary $\partial D_\epsilon$ differs greatly
\begin{equation}\nonumber
	\begin{split}
		\left(I+ \kappa \frac{\partial }{\partial \mathbf{n}}\right)(u^\epsilon-v_0) = f_D(\vartheta) -V_0(\mathbf{c}) + \kappa \nabla V_0(\mathbf{c})\cdot \left(\begin{array}{c}
\cos \vartheta \\ \sin \vartheta
\end{array} \right) - c_0\left(\mathcal{G}^{\kappa, \mathbf{c}}(\mathbf{c})-\frac{1}{2\pi}\log \epsilon  \right)  \\ -c_0\kappa \left(\frac{1}{2\pi \epsilon}- \nabla \mathcal{G}^{\kappa, \mathbf{c}}(\mathbf{c})\cdot \left(\begin{array}{c}
		\cos \vartheta \\ \sin \vartheta
		\end{array} \right) \right) + \epsilon g_D + \mathcal{O}(\epsilon(1+\kappa)).
	\end{split}
\end{equation}
Truncating this expansion to the leading order yields the boundary data for the exterior problem.
\subsection{Solution of the exterior problem}
We seek a bounded harmonic $w_0 \in C^{\infty}(\mathbb{R}^2\setminus D)$ which satisfies the following conditions
\begin{equation}\label{correct exterior decay condition}
\lim\limits_{|\xi|\rightarrow \infty}w_0\rightarrow 0,
\end{equation}
\begin{equation}\label{correct exterior boundary condition}
	\begin{split}
	w_0- \kappa \frac{\partial w_0}{\partial r}= f_D(\vartheta) -V_0(\mathbf{c}) + \kappa \nabla V_0(\mathbf{c})\cdot \left(\begin{array}{c}
\cos \vartheta \\ \sin \vartheta
\end{array} \right) - c_0\mathcal{G}^{\kappa, \mathbf{c}}(\mathbf{c}) \\ +\frac{c_0}{2\pi}\log \epsilon   - \frac{c_0\kappa}{2\pi \epsilon}+c_0\kappa \nabla \mathcal{G}^{\kappa, \mathbf{c}}(\mathbf{c})\cdot \left(\begin{array}{c}
		\cos \vartheta \\ \sin \vartheta
		\end{array} \right) .
	\end{split}
\end{equation}
Combining the decay condition, from above, and the result from section (\ref{limit theorem}) we know that $w_0$ decays to zero if and only if 
$$\frac{1}{2\pi}\int\limits_{\partial D_\epsilon}\left(w_0- \kappa \frac{\partial w_0}{\partial r} \right)\ dS_{\xi} = 0. $$
Substituting in the boundary conditions, we deduce that 
$$ \frac{1}{2\pi}\int\limits_{-\pi}^\pi f_D \ d\vartheta - V_0(\mathbf{c}) - c_0 \left(\mathcal{G}^{\kappa, \mathbf{c}}(\mathbf{c}) + \frac{\kappa}{2\pi \epsilon} - \frac{1}{2\pi}\log \epsilon\right)=0.$$
In the previous section this was a problem as the boundary data was a given function, however with the introduction of the point singularity in the interior problem we introduced an unknown variable $c_0$. Thus we may choose $c_0$ such that the above condition is always fulfilled, and consequently that the decay condition is always fulfilled
\begin{equation}\label{constant c0 definition}
	c_0 := \frac{\frac{1}{2\pi}\int\limits_{-\pi}^\pi f_D \ d\vartheta - V_0(\mathbf{c})}{\mathcal{G}^{\kappa, \mathbf{c}}(\mathbf{c}) + \frac{\kappa}{2\pi \epsilon} - \frac{1}{2\pi}\log \epsilon}.
\end{equation}
Substituting our constant $c_0$ into our boundary conditions and simplifying we deduce that $w_0$ satisfies
\begin{equation}\label{w boundary condition}
            w_0- \kappa \frac{\partial w_0}{\partial r} = \left(f_{D}(\vartheta) - \frac{1}{2\pi}\int\limits_{-\pi}^{\pi}f_{D}d\vartheta\right) + \kappa \left(\nabla V_0(\mathbf{c}) + c_0 \nabla\mathcal{G}^{\kappa, \mathbf{c}}(\mathbf{c})\right)\cdot \left(\begin{array}{c} \cos \vartheta \\ \sin \vartheta \end{array} \right), \quad \forall \vartheta \in [-\pi,\pi).
        \end{equation}At this point we recall that we made an assumption about the Fourier decomposition of the function $f_D$, substituting this into the above equation we deduce a Fourier decomposition of the boundary data
        \begin{equation}\label{Exterior Boundary data simplified}
            w_0- \kappa \frac{\partial w_0}{\partial r} = \sum\limits_{n=1}^\infty \left( \mathbf{f}_n + \delta_{n,1}\kappa \left(\nabla V_0(\mathbf{c}) + c_0 \nabla \mathcal{G}^{\kappa, \mathbf{c}}(\mathbf{c})\right)\right)\cdot \left(\begin{array}{c} \cos n\vartheta \\ \sin n\vartheta \end{array} \right), \quad \forall \vartheta \in [-\pi,\pi).
        \end{equation}
where $\mathbf{f}^n$ are components of the Fourier series decomposition of $f_D$ given in equation (\ref{Fourier series boundary data}). As we are working with a harmonic function in polar co-ordinates, we may now use the particular analytical form to deduce that
\begin{equation}\label{w0 interior analytic}
	w_0 = \sum\limits_{n=1}^{\infty}\frac{1}{r^n}\frac{1}{1+\kappa n}\left(\mathbf{f}^n + \kappa \delta_{n,1} \left(\nabla V_0(\mathbf{c}) + c_0 \nabla \mathcal{G}^{\kappa, \mathbf{c}}(\mathbf{c})\right) \right)\cdot \left(\begin{array}{c} \cos n\vartheta \\ \sin n\vartheta \end{array} \right),
\end{equation}
\subsection{Leading order approximation}
We may now express the leading order approximation to $u^\epsilon$ as the sum of the interior function $v_0$ and the exterior function $w_0$, the explicit expression is given by
\begin{equation}\nonumber
	\begin{split}
		u^{\epsilon}(\mathbf{x}) \approx u^\epsilon_0(\mathbf{x}) := V_0(\mathbf{x}) + c_0 G^\kappa(\mathbf{x}, \mathbf{c})  +\sum\limits_{n=1}^{\infty}\frac{\epsilon^n}{|\mathbf{x}-\mathbf{c}|^n}\frac{1}{1+\kappa n}\mathbf{f}^n\cdot \left(\begin{array}{c} \cos n\vartheta(\mathbf{x}-\mathbf{c}) \\ \sin n\vartheta(\mathbf{x}-\mathbf{c}) \end{array} \right) \\ + \frac{\epsilon}{|\mathbf{x}-\mathbf{c}|}\frac{\kappa}{1+\kappa }\left(\nabla V_0(\mathbf{c}) + c_0 \nabla \mathcal{G}^{\kappa, \mathbf{c}}(\mathbf{c})\right) \cdot \left(\begin{array}{c} \cos \vartheta(\mathbf{x}-\mathbf{c}) \\ \sin \vartheta(\mathbf{x}-\mathbf{c}) \end{array} \right) ,
	\end{split}
\end{equation}
for $\mathbf{x}\in \Omega_\epsilon$.  We now require a proof that the approximation $u^\epsilon_0$ is close to the exact solution $u^\epsilon$ and obtain a rate for the uniform bound between the two functions.
\section{Proof of the uniform bound}
In this section we shall prove that the difference between the exact solution and asymptotic approximation is uniformly bounded
\begin{equation}\nonumber
	|u^\epsilon-u^\epsilon_0|(\mathbf{x}) = \mathcal{O}\left(\epsilon(1+\kappa)\left(1+\frac{\kappa+1}{\frac{\kappa}{\epsilon}-\log \epsilon}\right) \right), \quad \forall \mathbf{x}\in \overline{\Omega}_\epsilon.
\end{equation}
We know from the previously derived maximum principle, that the difference between the exact solution $u^\epsilon$ and our asymptotic approximation $u^\epsilon_0$ is uniformly bounded by the maximum difference of the Robin boundary data. Thus 
$$|u^\epsilon - u^\epsilon_0|(\mathbf{x}) \leqslant \sup\limits_{\mathbf{y}\in \partial \Omega_\epsilon}\left| \left(I + \kappa \frac{\partial }{\partial \mathbf{n}} \right)(u^\epsilon-u^\epsilon_0)(\mathbf{y})\right|, $$
thus our objective in this section will be to derive an upper bound on the above trace in terms of constants, $\epsilon$ and $\kappa$; which will allow us to determine at which points the approximation breaks down. However, we must derive a series of lemma's which will be used to prove this bound.
           \subsection{Local bound of a harmonic function\label{section: Local bound of a harmonic function}}
                Let $a\in C^\infty(\Omega)\cap C(\overline{\Omega})$ be harmonic, consider a ball of radius $R>0$ centred at $\mathbf{x}_0 \in \Omega$ which is such that $\overline{B}_R(\mathbf{x}_0) \subset \overline{\Omega}$. We obtain the following bound
                \begin{equation}\nonumber
                    \left|a(\mathbf{x})-a(\mathbf{x}_0) \right|\leqslant \frac{2|\mathbf{x}-\mathbf{x}_0|}{R-|\mathbf{x}-\mathbf{x}_0|}\left(a(\mathbf{x}_0)-\inf_{\mathbf{x}\in \overline{\Omega}}a(\mathbf{x}) \right), \quad \forall \mathbf{x}\in B_{R}(\mathbf{x}_0).
                \end{equation}
            \begin{proof}
                We define the function $h := a - \inf\limits_{\mathbf{y}\in \overline{\Omega}}a(\mathbf{y})\geqslant 0$  on $\overline{\Omega}$, this function is harmonic thus we may apply Harnack's inequality to obtain
                \begin{equation}\nonumber
                    \frac{R-|\mathbf{x}-\mathbf{x}_0|}{R+|\mathbf{x}-\mathbf{x}_0|}h(\mathbf{x}_0) \leqslant h(\mathbf{x}) \leqslant \frac{R+|\mathbf{x}-\mathbf{x}_0|}{R-|\mathbf{x}-\mathbf{x}_0|}h(\mathbf{x}_0), \quad \forall \mathbf{x}\in B_R(\mathbf{x}_0)
                \end{equation}
                \begin{equation}\nonumber
                    -\frac{2|\mathbf{x}-\mathbf{x}_0|}{R+|\mathbf{x}-\mathbf{x}_0|}h(\mathbf{x}_0) \leqslant h(\mathbf{x})-h(\mathbf{x}_0) \leqslant \frac{2|\mathbf{x}-\mathbf{x}_0|}{R-|\mathbf{x}-\mathbf{x}_0|}h(\mathbf{x}_0), \quad \forall \mathbf{x}\in B_R(\mathbf{x}_0)
                \end{equation}
                The definition of $h$ implies that$$h(\mathbf{x})-h(\mathbf{x}_0) =\left(a(\mathbf{x}) - \inf_{\mathbf{x}\in \overline{\Omega}}a(\mathbf{x})\right) -\left(a(\mathbf{x}_0) - \inf_{\mathbf{x}\in \overline{\Omega}}a(\mathbf{x})\right) =a(\mathbf{x}) - a(\mathbf{x}_0)$$ we have that
                \begin{equation}\nonumber
                    -\frac{2|\mathbf{x}-\mathbf{x}_0|}{R+|\mathbf{x}-\mathbf{x}_0|}\left(a(\mathbf{x}_0) - \inf_{\mathbf{x}\in \overline{\Omega}}a(\mathbf{x}) \right)\leqslant a(\mathbf{x})-a(\mathbf{x}_0) \leqslant \frac{2|\mathbf{x}-\mathbf{x}_0|}{R-|\mathbf{x}-\mathbf{x}_0|}\left(a(\mathbf{x}_0) - \inf_{\mathbf{x}\in \overline{\Omega}}a(\mathbf{x}) \right),
                \end{equation}
                for all $\mathbf{x}\in B_R(\mathbf{x}_0)$.
			\end{proof}
            \subsection{Local bound on the gradient of a harmonic function\label{section: Local bound on the radial derivative of a harmonic function}}
                As in the previous section we shall consider a harmonic $a\in C^\infty(\Omega)\cap C(\overline{\Omega})$, and a ball of radius $R>0$ centred at $\mathbf{x}_0 \in \Omega$ which is such that $\overline{B}_R(\mathbf{x}_0) \subset \overline{\Omega}$. We obtain the following bounds
                \begin{equation}\nonumber
                    \left|\nabla a \right|(\mathbf{x}_0)\leqslant \frac{2}{R}\|a\|_{L^\infty(\Omega)}.
                \end{equation}
                \begin{equation}\nonumber
                    \left|\nabla a(\mathbf{x})-\nabla a(\mathbf{x}_0) \right|\leqslant \frac{2\|a\|_{L^\infty(\Omega)}}{R} \frac{|\mathbf{x}-\mathbf{x}_0|(2R-|\mathbf{x}-\mathbf{x}_0|)}{(R-|\mathbf{x}-\mathbf{x}_0|)^2},
                \end{equation}
                for all $ |\mathbf{x}-\mathbf{x}_0|< R$.
            \begin{proof}
                We define the function $\alpha\in C^\infty([0,R]\times [-\pi,\pi))$ by
                \begin{equation}\nonumber
                	\alpha(r,\vartheta) := a\left(\mathbf{x}_0 + r \left(\begin{array}{c} \cos \vartheta \\ \sin \vartheta \end{array} \right) \right), \ \Rightarrow  \ \nabla a\left(\mathbf{x}_0 + r \left(\begin{array}{c} \cos \vartheta \\ \sin \vartheta \end{array} \right) \right) \cdot  \left(\begin{array}{c} \cos \vartheta \\ \sin \vartheta \end{array} \right) =  \frac{\partial \alpha}{\partial r}(r,\vartheta).
                \end{equation}
               We may express the harmonic function as the following analytic formula
               $$\alpha(r,\vartheta) = \frac{a^0}{2}+\sum\limits_{n=1}^\infty r^n \left(a^n_1 \cos n \vartheta + a^n_2 \sin n\vartheta \right),$$
               where the constants $a^n_1 :=\frac{1}{\pi R^n} \int\limits_{-\pi}^{\pi}\alpha (R,\vartheta)\cos n \vartheta\ d\vartheta$ and $a^n_2 :=\frac{1}{\pi R^n} \int\limits_{-\pi }^{\pi}\alpha (R,\vartheta)\sin n \vartheta\ d\vartheta$. We consider the norm of the derivative of $a$
               $$\left|\nabla a\left(\mathbf{x}_0 \right)\right| =  \left(\left|\frac{\partial \alpha}{\partial r}\right|+ \frac{1}{r}\left|\frac{\partial \alpha}{\partial \vartheta}\right|\right)_{r=0}= \sqrt{\left(a^1_1 \cos \vartheta + a^1_2 \sin\vartheta \right)^2+\left(a^1_1 \cos \vartheta + a^1_2 \sin \vartheta \right)^2}= |\mathbf{a}^1|.$$
                It is clear by the definition of the Fourier coefficients that $|\mathbf{a}^n|\leqslant \frac{2}{R^n}\|a\|_\infty$ thus we immediately obtain our first result. To derive our second result, we consider the difference of derivatives
                \begin{equation}\nonumber
                	\frac{\partial \alpha}{\partial r}(r,\vartheta) - \frac{\partial \alpha }{\partial r}(0,\vartheta) =\sum\limits_{n=2}^\infty n r^{n-1}\left(a^n_1 \cos n \vartheta + a^n_2 \sin n\vartheta \right),
                \end{equation}
                \begin{equation}\nonumber
                	\frac{\partial \alpha}{\partial \vartheta}(r,\vartheta) - \frac{\partial \alpha }{\partial \vartheta}(0,\vartheta) =\sum\limits_{n=2}^\infty n r^{n}\left(a^n_2 \cos n \vartheta - a^n_1 \sin n\vartheta \right),
                \end{equation}
Applying the triangle inequality and the gradient in polar co-ordinates we deduce that 
$$|\nabla a(\mathbf{x})- \nabla a(\mathbf{x}_0)|= \left|\frac{\partial \alpha}{\partial r}(r,\vartheta) - \frac{\partial \alpha }{\partial r}(0,\vartheta)\right|+\frac{1}{r}\left|\frac{\partial \alpha}{\partial \vartheta}(r,\vartheta) - \frac{\partial \alpha }{\partial \vartheta}(0,\vartheta)\right|\leqslant \sum\limits_{n=2}^\infty n r^{n-1}|\mathbf{a}^n|.$$
Using the previous inequality on Fourier coefficients and the derivative of a geometric sequence we deuce that
\begin{equation}\nonumber
                	|\nabla a(\mathbf{x})- \nabla a(\mathbf{x}_0)|\leqslant \frac{2\|a\|_\infty}{R}\sum\limits_{n=2}^\infty n \left(\frac{r}{R}\right)^{n-1} = \frac{2\|a\|_\infty}{R} \frac{r(2R-r)}{(R-r)^2}.
                \end{equation}
                               
			\end{proof}
            
            \subsection{A bound on the discrepancy on $\partial D_\epsilon$}
Recall that our maximum principle has bounded the difference between the exact solution $u^\epsilon$ and our leading order approximation $u^\epsilon_0$ by the maximum difference in the Robin trace. We define the Robin boundary operator $\mathcal{T}^\kappa := I+\kappa \frac{\partial }{\partial \mathbf{n}}$and consider the difference on the boundary $\partial D_\epsilon$
\begin{equation}\nonumber
	\mathcal{T}^\kappa(u^\epsilon-u^\epsilon_0) = (f_D + \epsilon g_D) - \mathcal{T}^\kappa V_0 - c_0 \mathcal{T}^\kappa G^\kappa(\cdot, \mathbf{c}) - \mathcal{T}^\kappa w_0.
\end{equation}
Recall the boundary conditions of the function $w_0$ from equation (\ref{Exterior Boundary data simplified})
\begin{equation}\nonumber
	\mathcal{T}^\kappa(u^\epsilon-u^\epsilon_0) = \frac{1}{2\pi}\int\limits_{-\pi}^{\pi}f_D\ d\vartheta + \epsilon g_D(\mathbf{x}) - \mathcal{T}^\kappa V_0(\mathbf{x}) - c_0 \mathcal{T}^\kappa G^\kappa(\mathbf{x}, \mathbf{c}) -\kappa \frac{\partial V_0}{\partial r}(\mathbf{c}) -\kappa c_0 \frac{\partial \mathcal{G}^{\kappa, \mathbf{c}}}{\partial r}(\mathbf{c}) .
\end{equation}
where the notation $\frac{\partial V_0}{\partial r}(\mathbf{c}) :=\nabla V_0(\mathbf{c})\left(\begin{array}{c} \cos\vartheta \\ \sin \vartheta \end{array} \right)$ and  $\frac{\partial \mathcal{G}^{\kappa, \mathbf{c}}}{\partial r}(\mathbf{c}) :=\nabla  \mathcal{G}^{\kappa, \mathbf{c}}(\mathbf{c})\left(\begin{array}{c} \cos\vartheta \\ \sin \vartheta \end{array} \right)$. To simplify the above expression we substitute the decomposition of the function $G^\kappa$, apply the fact on the boundary of $D_\epsilon$ the operator $\mathcal{T}^\kappa := I-\kappa \frac{\partial }{\partial r}$ to deduce that 
$$\mathcal{T}^\kappa G^\kappa(\mathbf{x}, \mathbf{c}) = \mathcal{G}^{\kappa, \mathbf{c}}(\mathbf{c}) -\frac{1}{2\pi}\log \epsilon + \frac{\kappa}{2\pi \epsilon} + \left(\mathcal{G}^{\kappa, \mathbf{c}}(\mathbf{x})-\mathcal{G}^{\kappa, \mathbf{c}}(\mathbf{c}) \right) - \kappa \frac{\partial \mathcal{G}^{\kappa, \mathbf{c}}}{\partial r}(\mathbf{x}),$$
on the boundary $\partial D_\epsilon$. Substituting the above expression into the discrepancy yields 
\begin{equation}\nonumber
	\begin{split}
		\mathcal{T}^\kappa(u^\epsilon-u^\epsilon_0) = 
        \frac{1}{2\pi}\int\limits_{-\pi}^{\pi}f_D\ d\vartheta - V_0(\mathbf{x})-c_0\left(\mathcal{G}^{\kappa, \mathbf{c}}(\mathbf{c}) -\frac{1}{2\pi}\log \epsilon + \frac{\kappa}{2\pi \epsilon}\right) + \kappa \left(\frac{\partial V_0}{\partial r}(\mathbf{x})-\frac{\partial V_0}{\partial r}(\mathbf{c}) \right)
       \\ -c_0 \left(\mathcal{G}^{\kappa, \mathbf{c}}(\mathbf{x})-\mathcal{G}^{\kappa, \mathbf{c}}(\mathbf{c}) \right) 
        +c_0 \kappa \left(\frac{\partial \mathcal{G}^{\kappa, \mathbf{c}}}{\partial r}(\mathbf{x}) - \frac{\partial \mathcal{G}^{\kappa, \mathbf{c}}}{\partial r}(\mathbf{c}) \right)+\epsilon g_D
	\end{split}
\end{equation}
To simplify the $\mathcal{G}^{\kappa, \mathbf{c}}(\mathbf{c}) -\frac{1}{2\pi}\log \epsilon + \frac{\kappa}{2\pi \epsilon}$ term we substitute the definition of the constant $c_0$ from equation (\ref{constant c0 definition}) to obtain
\begin{equation}\nonumber
	\begin{split}
		\mathcal{T}^\kappa(u^\epsilon-u^\epsilon_0) = 
        \left(V_0(\mathbf{c}) - V_0(\mathbf{x})\right) -c_0 \left(\mathcal{G}^{\kappa, \mathbf{c}}(\mathbf{x})-\mathcal{G}^{\kappa, \mathbf{c}}(\mathbf{c}) \right)   \\+ \kappa \left(\frac{\partial V_0}{\partial r}(\mathbf{x})-\frac{\partial V_0}{\partial r}(\mathbf{c}) \right)
        +c_0 \kappa \left(\frac{\partial \mathcal{G}^{\kappa, \mathbf{c}}}{\partial r}(\mathbf{x}) - \frac{\partial \mathcal{G}^{\kappa, \mathbf{c}}}{\partial r}(\mathbf{c}) \right) + \epsilon g_D
	\end{split}
\end{equation}
Thus the discrepancy of the approximation on the boundary $\partial D_\epsilon$ is given by the difference of harmonic functions and their evaluation at specific points. We may use the triangle inequality, the maximum principle from section (\ref{Bounded Robin Maximum Principle}) and the bounds derived in sections \ref{section: Local bound of a harmonic function} and \ref{section: Local bound on the radial derivative of a harmonic function}, to deduce that 
\begin{equation}\nonumber
	\begin{split}
		\left|\mathcal{T}^\kappa(u^\epsilon-u^\epsilon_0)\right| \leqslant 
        \frac{4\epsilon \|f_\Omega\|_\infty}{R_{\min}-\epsilon} + \epsilon\|g_D\|_\infty +|c_0| \frac{2\epsilon} {\pi(R_{\min}-\epsilon)}\left(\mathcal{G}^{\kappa}_\infty + \frac{\kappa}{R_{\min}} \right)     \\+ \frac{2\kappa}{R_{\min}}\frac{\epsilon(2R_{\min}-\epsilon)}{(R_{\min}-\epsilon)^2}\|f_\Omega\|_\infty +|c_0| \kappa \left(\mathcal{G}^{\kappa}_\infty  + \frac{\kappa}{R_{\min}} \right) \frac{2\kappa}{R_{\min}}\frac{\epsilon(2R_{\min}-\epsilon)}{2\pi (R_{\min}-\epsilon)^2}
	\end{split}
\end{equation}
where $\mathcal{G}^{\kappa}_\infty := \max\left\{|\log R_{\min}|, |\log R_{\max}| \right\}$. To remove the $R_{\min}-\epsilon$ terms we apply the geometrical assumption from equation (\ref{epsilon bound}) to deduce the uniform bound
\begin{equation}\nonumber
	|\mathcal{T}^\kappa(u^\epsilon-u^\epsilon_0)|(\mathbf{x}) = \mathcal{O}\left(\epsilon(1+\kappa)\left(1+\frac{\kappa+1}{\frac{\kappa}{\epsilon}-\log \epsilon}\right) \right), \quad \forall \mathbf{x}\in \partial D_\epsilon.
\end{equation}
\subsection{A bound on the discrepancy on $\partial \Omega$ }            
We now consider the discrepancy between the exact solution and the leading order approximation on the boundary of our interior domain, 
\begin{equation}\nonumber
	\mathcal{T}^\kappa (u^\epsilon-u^\epsilon_0) = (f_\Omega + \epsilon g_\Omega) - \mathcal{T}^\kappa V_0 - c_0 \mathcal{T}^\kappa G^\kappa(\cdot, \mathbf{c}) - \mathcal{T}^\kappa w_0.
\end{equation}
We recall the boundary conditions of the functions $V_0$, from equation (\ref{Interior boundary condition}) and $G^\kappa$ , from equation (\ref{Greens function definition}), we deduce that 
\begin{equation}\nonumber
	\mathcal{T}^\kappa (u^\epsilon-u^\epsilon_0) = \epsilon g_\Omega(\mathbf{x}) - w_0(\mathbf{x}) - \kappa\frac{\partial w_0}{\partial \mathbf{n}}(\mathbf{x}), \quad \forall \mathbf{x}\in \partial \Omega.
\end{equation}
Thus we seek a uniform bound for both $w_0$ and $\frac{\partial w_0}{\partial \mathbf{n}}$, this analysis will rely on the expansion from equation (\ref{w0 interior analytic}), 
\begin{equation}\nonumber
	\begin{split}
		w_0(\mathbf{x}) = \frac{\kappa}{1+\kappa}\frac{\epsilon}{|\mathbf{x}-\mathbf{c}|} \left(\nabla V_0(\mathbf{c}) + c_0 \nabla \mathcal{G}^{\kappa, \mathbf{c}}(\mathbf{c})\right) \cdot \left(\begin{array}{c} \cos \vartheta(\mathbf{x}-\mathbf{c}) \\ \sin \vartheta(\mathbf{x}-\mathbf{c}) \end{array} \right)\\ + \sum\limits_{n=1}^\infty \frac{\epsilon^n}{|\mathbf{x}-\mathbf{c}|^n}\frac{1}{1+\kappa n}\mathbf{f}^n \cdot \left(\begin{array}{c} \cos n\vartheta(\mathbf{x}-\mathbf{c}) \\ \sin n\vartheta(\mathbf{x}-\mathbf{c}) \end{array} \right).
	\end{split}
\end{equation}
Using the triangle inequality, the bounds from equation (\ref{section: Local bound on the radial derivative of a harmonic function}) and the geometric inequality from equation (\ref{Rmin and Rmax definition}) we deduce that for $\mathbf{x}\in \partial \Omega$
$$|w_0|(\mathbf{x}) \leqslant \frac{\kappa}{1+\kappa}\frac{\epsilon}{R_{\min}} \left(\frac{2}{R_{\min}}\|f_\Omega\|_{\infty} + \frac{|c_0|}{\pi R_{\min}}\left( \mathcal{G}^{\kappa}_\infty+ \frac{\kappa}{R_{\min}} \right)\right) + \sum\limits_{n=1}^\infty \left(\frac{\epsilon}{R_{\min}}\right)^n\frac{|\mathbf{f}^n|}{1+\kappa n}. $$
To bound the summation term, we use the face that $1+\kappa \leqslant 1+\kappa n$, the bound on the Fourier terms $|\mathbf{f}^n|\leqslant 2\|f_D\|_\infty$ and the sum of a geometric series to deduce that
$$ \sum\limits_{n=1}^\infty \left(\frac{\epsilon}{R_{\min}}\right)^n\frac{|\mathbf{f}^n|}{1+\kappa n} \leqslant \frac{2\|f_D\|_\infty}{1+\kappa}\frac{\epsilon}{R_{\min}-\epsilon}. $$
We now consider a bound on a normal derivative by applying the triangle inequality and the geometric assumptions from equation (\ref{Rmin and Rmax definition})
$$|\nabla w_0| \leqslant \frac{\kappa}{1+\kappa}\frac{\epsilon}{R_{\min}^2} \left(|\nabla V_0|(\mathbf{c}) + |c_0| |\nabla \mathcal{G}^{\kappa, \mathbf{c}}|(\mathbf{c})\right) + \sum\limits_{n=1}^\infty \frac{\epsilon^n}{R_{\min}^{n+1}}\frac{n|\mathbf{f}^n|}{1+\kappa n}.$$
We have that $\frac{1}{1+\kappa}\leqslant \frac{n}{1+\kappa n}\leqslant \frac{1}{\kappa}$ for $n \in \mathbb{N}$ and the bound on the harmonic gradient from section \ref{section: Local bound on the radial derivative of a harmonic function} to deduce that
$$|\nabla w_0|\leqslant \frac{\kappa}{1+\kappa} \frac{\epsilon}{R_{\min}^3} \left(2\|f_\Omega\|_{\infty} + \frac{|c_0|}{\pi} \left(\mathcal{G}^\kappa_\infty + \frac{\kappa}{R_{\min}} \right) \right) + \frac{2}{\kappa} \|f_D\|_{\infty} \frac{\epsilon}{R_{\min}-\epsilon}.$$
Thus we have the uniform bound
\begin{equation}\nonumber
	\begin{split}
		|\mathcal{T}^\kappa (u^\epsilon-u^\epsilon_0)|(\mathbf{x})\leqslant  \frac{\kappa^2\epsilon}{1+\kappa}\frac{2\|f_{\Omega}\|_{\infty} }{R_{\min}^3} + \frac{1}{\pi R_{\min}^3}  \frac{\left(\mathcal{G}^\kappa_\infty + \frac{\kappa}{R_{\min}} \right)}{1+\kappa}|c_0|\epsilon\kappa^2 +\left(\frac{ 2\|f_D\|_{\infty}}{R_{\min}-\epsilon} + \|g_{\Omega}\|_\infty\right)\epsilon \\ + \frac{1}{\pi R_{\min}}  \frac{\left(\mathcal{G}^\kappa_{\infty} + \frac{\kappa}{R_{\min}} \right)}{1+\kappa} |c_0|\epsilon\kappa + \frac{\kappa \epsilon}{1+\kappa} \frac{2\|f_\Omega\|_{\infty}}{R_{\min}^2} +\frac{2\|f_D\|_{\infty}}{R_{\min}-\epsilon} \frac{\epsilon}{1+\kappa},
	\end{split}
\end{equation}
for all $\mathbf{x}\in \partial \Omega$. Finally we use the geometric assumptions from equation (\ref{epsilon bound}), to remove the $R_{\min}-\epsilon$ terms, in addition to the inequality $\frac{\kappa^2}{1+\kappa}\leqslant \kappa$ for all $\kappa \geqslant 0$ to deduce that
\begin{equation}\nonumber
	|\mathcal{T}^\kappa(u^\epsilon-u^\epsilon_0)|(\mathbf{x}) = \mathcal{O}\left(\epsilon(1+\kappa)\left(1+\frac{\kappa+1}{\frac{\kappa}{\epsilon}-\log \epsilon}\right) \right), \quad \forall \mathbf{x}\in \partial \Omega.
\end{equation}
\section{Conclusion}
We have derived a compound asymptotic approximation to a harmonic function in a singular domain with Robin boundary conditions. We have proven a uniform bound between our exact solution and our approximation which is of order $$ \epsilon(1+\kappa)\left(1+\frac{\kappa+1}{\frac{\kappa}{\epsilon}-\log \epsilon}\right),$$
where $\kappa >0$ is the extrapolation length \cite{Extrapolation_length}. To prove this error we derived a maximum principle for Robin boundary conditions and . The approximation is written as the sum of three terms: an interior solution which satisfies the boundary conditions on $\partial \Omega$, a Green's function centred on the inclusion and an exterior solution which satisfies the boundary conditions on $\partial D_\epsilon$ whilst compensating for the interior solution and the Green's function. As we have considered a circular inclusion, the exterior solution has an explicit analytical expansion, in terms of the Fourier coefficients of the boundary data. 

This is known as a compound asymptotic expansion as  we may consider the next order approximation $u^\epsilon_1$ by considering the function $u^\epsilon_1 \approx \frac{1}{\epsilon}(u^\epsilon-u^\epsilon_0)$. The function $\frac{1}{\epsilon}(u^\epsilon-u^\epsilon_0)$ satisfies Laplace's equation in $\Omega_\epsilon$ and Robin boundary conditions, to obtain $u_1^\epsilon$ we consider the leading order term of the boundary data of $\frac{1}{\epsilon}(u^\epsilon-u^\epsilon_0)$, and then we apply the same method we used to derive $u^\epsilon_0$. Thus we have an algorithm to obtain higher order terms as we may approximate $u^\epsilon_{N+1}:= \frac{1}{\epsilon^N}\left(u^\epsilon-\sum\limits_{n=0}^N\epsilon^n u^\epsilon_n\right)$ using the same method.

One of the many applications of such an approximation would be in numerical analysis, the finite element method approximates the surface of a domain by a series of triangular tiles; thus for a small inclusion the size of the triangular tiles would be small and thus the number of triangles would be correspondingly high, which drives up computational costs. However, the leading order term only requires numerically approximating $V_0$ and $G^\kappa$ both of which are functions on the bulk domain $\Omega$, which would allow the use of a coarse mesh. 

Additionally, whilst a change in the size of the inclusion would require a new triangular mesh to be generated for each distinct $\epsilon$, the asymptotic formula we have derived requires no such re-computation as the exterior solution has an explicit dependence on $\epsilon$. Thus, for a fixed inclusion centre $\mathbf{c}\in \Omega$, if an numerical analyst wished to compare inclusions of multiple sizes then using the leading order approximation would be ideal.

However, this approximation is useful if one wishes to analyse different inclusion placements. The optimal location of the inclusion would correspond to the minimum energy, thus in future work we may consider the derivative of the energy with respect to the inclusion centre $$\frac{\partial }{\partial \mathbf{c}}\left(E_\kappa \left[\left(\begin{array}{c} \cos u^\epsilon_0 \\ \sin u^\epsilon_0 \end{array} \right)\right]\right) = \underline{0}.$$ 
to approximate the minimising location.

The work of Maz'ya, Nazarov and Plamenevskij  \cite{Nazarov}, they considered an asymptotic approximation in a singular domain with either Dirichlet or Neumann boundary conditions. They did not consider Robin boundary conditions, however the scope of their work was far wider, as they investigated:
\begin{itemize}
\item A more general class of inclusions, rather than just the circular that we have considered in this piece.
\item They considered systems of dimensions higher than two.
\item The Helmholtz equation, in addition to Laplace's equation.
\end{itemize}
The analysis of a more general inclusion shape is trivial, the conformal map of a harmonic function is harmonic\cite{CourantR2012DPCM}, thus the above asymptotic approximation is valid for any inclusion which can be mapped (in such a manner) to a circle. The analysis in higher dimensions is not a topic we have considered, however in Maz'ya, Nazarov and Plamenevskij's work the compound asymptotic expansion for the three dimensional system was far less complex than the two dimensional one. This is because in three dimensions the decay condition is automatically satisfied, thus rendering the Green's function term obsolete, we suspect that the analysis for Robin boundary conditions would be similar.

Finally, we note the point at which the approximation breaks down. As $\epsilon$ increases so to does the uniform bound, this is expected as the approximation is made under the assumption of a small inclusion. However, the uniform bound increases as the extrapolation length $\kappa \rightarrow \infty$. In the context of our boundary value problem this corresponds a change in the boundary conditions, from inhomogenous Robin to homogenous Neumann. 
\section{Acknowledgements}
We are grateful for the financial support of the University of Bath's department of Mathematical Sciences. Additionally, we are thankful for the support and guidance of Apala Majumdar whose guidance in the aspects of liquid crystals is invaluable. 
\printbibliography
\end{document}